\documentclass[11pt,draft]{amsart}

\usepackage{cite,amsmath,amssymb,amsthm,latexsym,a4,graphicx,gastex}

\usepackage[cp1251]{inputenc}

\newtheorem{theorem}{Theorem}[section]

\newtheorem{proposition}[theorem]{Proposition}

\newtheorem{lemma}[theorem]{Lemma}

\newtheorem{corollary}[theorem]{Corollary}

\DeclareMathOperator{\con}{con}

\DeclareMathOperator{\Perm}{Perm}

\DeclareMathOperator{\Stab}{Stab}

\DeclareMathOperator{\Sub}{Sub}

\numberwithin{equation}{section}

\makeatletter

\renewcommand*\subjclass[2][2010]{\def\@subjclass{#2}\@ifundefined{subjclassname@#1}{\ClassWarning{\@classname}{Unknown edition (#1) of Mathematics Subject Classification; using '2010'.}}{\@xp\let\@xp\subjclassname\csname subjclassname@#1\endcsname}}

\makeatother

\begin{document}

\title[Cancellable elements of the lattice of semigroup varieties]{Cancellable elements of the lattice\\
of semigroup varieties: varieties satisfying\\
a permutational identity of length 3}

\author{B.~M.~Vernikov}

\address{Ural Federal University, Institute of Natural Sciences and Mathematics, Lenina 51, 620083 Ekaterinburg, Russia}

\email{bvernikov@gmail.com}

\date{}

\thanks{The work was partially supported by the Ministry of Education and Science of the Russian Federation (project 1.6018.2017/8.9) and by the Russian Foundation for Basic Research (grant No.\,17-01-00551).}

\begin{abstract}
We completely determine all semigroup varieties satisfiyng a permutational identity of length 3 that are cancellable elements of the lattice of all semigroup varieties. Using this result, we provide a series of new examples of semigroup varieties that are modular but not cancellable elements of this lattice.
\end{abstract}



\maketitle

\section{Introduction and summary}
\label{intr}

There are a number of articles devoted to an examination of special elements in the lattice $\mathbb{SEM}$ of all semigroup varieties (see surveys~\cite[Section~14]{Shevrin-Vernikov-Volkov-09} and~\cite{Vernikov-15}).One can recall definitions of two types of special elements that appear below. An element $x$ of a lattice $\langle L;\vee,\wedge\rangle$ is called
\begin{align*}
&\text{\emph{modular} if}\enskip&&(\forall y,z\in L)\quad \bigl(y\le z\longrightarrow(x\vee y)\wedge z=(x\wedge z)\vee y\bigr),\\
&\text{\emph{cancellable} if}\enskip&&(\forall y,z\in L)\quad (x\vee y=x\vee z\ \&\ x\wedge y=x\wedge z\longrightarrow y=z).
\end{align*}
It is evident that every cancellable element of a lattice is modular. A valuable information about modular and cancellable elements in abstract lattices can be found in~\cite{Seselja-Tepavcevic-01}, for instance.

Several results about modular elements of the lattice $\mathbb{SEM}$ were provided in the papers~\cite{Jezek-McKenzie-93,Shaprynskii-12,Vernikov-07}, while cancellable elements of $\mathbb{SEM}$ were considered in~\cite{Gusev-Skokov-Vernikov-18+,Skokov-Vernikov-18+}. In particular, commutative semigroup varieties that are modular elements in $\mathbb{SEM}$ are completely determined by the author  in~\cite[Theorem~3.1]{Vernikov-07}. Further, it is verified by S.Gusev, D.Skokov and the author in~\cite[Theorem~1.1]{Gusev-Skokov-Vernikov-18+} that the properties to be modular and cancellable elements of $\mathbb{SEM}$ are equivalent in the class of commutative semigroup varieties. The question is naturally arised, whether this equivalence true for arbitrary semigroup varieties. The negative answer to this question is given in~\cite{Skokov-Vernikov-18+}. Recall that an identity of the form
\begin{equation}
\label{permut id}
x_1x_2\cdots x_n\approx x_{1\pi}x_{2\pi}\cdots x_{n\pi}
\end{equation}
where $\pi$ is a non-trivial permutation on the set $\{1,2,\dots,n\}$ is called a \emph{permutational} identity. The number $n$ is called a \emph{length} of the identity~\eqref{permut id}. Clearly, the commutative law is a permutational identity of length~2. In~\cite[Theorem~1.1]{Skokov-Vernikov-18+} (see Proposition~\ref{mod-permut-3} below) modular elements of the lattice $\mathbb{SEM}$ are described within the class of varieties satisfying a permutational identity of length~3. After that, among varieties satisfying the conclusion of this statement, a variety is found that is not a cancellable element in $\mathbb{SEM}$~\cite[Proposition~1.2]{Skokov-Vernikov-18+}.

Here we continue examinationes started in~\cite{Gusev-Skokov-Vernikov-18+,Skokov-Vernikov-18+}. Namely, we completely classify varieties that satisfy a  permutational identity of length~3 and are cancellable elements in $\mathbb{SEM}$. Comparison of this result and~\cite[Theorem~1.1]{Skokov-Vernikov-18+} permits to specify a number of new examples of modular but not cancellable elements in $\mathbb{SEM}$. 

A semigroup variety is called a \emph{nil-variety} if it consists of nilsemigroups. Semigroup words unlike letters are written in bold. Two sides of identities we connect by the symbol~$\approx$, while the symbol~$=$ stands for the equality relation on the free semigroup. As usual, we write the pair of identities $x\mathbf{u\approx u}x\approx\mathbf u$ where the letter $x$ does not occur in the word \textbf u in the short form $\mathbf u\approx 0$ and refer to the expression $\mathbf u\approx 0$ as to a single identity.  We denote by \textbf T the trivial semigroup variety and by \textbf{SL} the variety of all semilattices.

The main result of the article is the following

\begin{theorem}
\label{canc-permut-3}
A semigroup variety $\mathbf V$ satisfying a permutational identity of length $3$ is a cancellable element in the lattice $\mathbb{SEM}$ if and only if $\mathbf V=\mathbf{M\vee N}$ where $\mathbf M$ is one of the  varieties $\mathbf T$ or $\mathbf{SL}$, while the variety $\mathbf N$ satisfies the following identities:
\begin{align}
\label{xyz=yxz=xzy}xyz&\approx yxz\approx xzy,\\
\label{xxy=0}x^2y&\approx 0.
\end{align}
\end{theorem}

It is easy to see that if a variety satisfies the identities~\eqref{xyz=yxz=xzy} then it satisfies all permutational identities of length~3. Thus, Theorem~\ref{canc-permut-3} shows that if a cancellable element of $\mathbb{SEM}$ satisfies one permutational identity of length~3 then it satisfies all identities of such a form. As we will seen below, the analog of this claim is true for permutational identities of arbitrary length (see Proposition~\ref{alternative for cancellable}).

The article consists of three sections. Section~\ref{prel} contains auxiliary results. In Section~\ref{proof} we verify Theorem~\ref{canc-permut-3}.

\section{Preliminaries}
\label{prel}

First of all, we reproduce the main result of the article~\cite{Gusev-Skokov-Vernikov-18+}.

\begin{proposition}[\textup{\!\!\cite[Theorem~1.1]{Gusev-Skokov-Vernikov-18+}}]
\label{mod-permut-3}
A semigroup variety $\mathbf V$ satisfying a permutational identity of length $3$ is a modular element in the lattice $\mathbb{SEM}$ if and only if $\mathbf V=\mathbf{M\vee N}$ where $\mathbf M$ is one of the  varieties $\mathbf T$ or $\mathbf{SL}$, while the variety $\mathbf N$ satisfies one of the following identity systems: \textup{(i)}~$xyz\approx zyx$, $x^2y\approx 0$; \textup{(ii)}~$xyz\approx yzx$, $x^2y\approx 0$; \textup{(iii)}~$xyz\approx yxz$, $xyzt\approx xzty$, $xy^2\approx 0$; \textup{(iv)}~$xyz\approx xzy$, $xyzt\approx yzxt$, $x^2y\approx 0$.\qed
\end{proposition}

Comparison of Theorem~\ref{canc-permut-3} and Proposition~\ref{mod-permut-3} allows us to provide many examples of modular but not cancellable elements in the lattice $\mathbb{SEM}$. In particular, the varieties given by each of the identity systems (i)--(iv)  listed in Proposition~\ref{mod-permut-3} have this property.

\begin{lemma}
\label{modular non-cancellable}
\!\!\textup{(\!\!\cite[Lemma~2.3]{Gusev-Skokov-Vernikov-18+})} Let $x$ be a modular but not cancellable element of a lattice $L$ and let $y$ and $z$ be different elements of $L$ such that $x\vee y=x\vee z$ and $x\wedge y=x\wedge z$. Then there is an element $x'\in L$ such that $x'\le x$, $x'\vee y=x'\vee z$, $x'\wedge y=x'\wedge z$ and $y\vee z=x'\vee y$.\qed
\end{lemma}

The following lemma immediately follows from Lemmas~2.1 and~2.8 of the article~\cite{Gusev-Skokov-Vernikov-18+}.

\begin{lemma}
\label{join with SL}
A semigroup variety $\mathbf X$ is a cancellable element of the lattice $\mathbb{SEM}$ if and only if the variety $\mathbf{X\vee SL}$ has this property.\qed
\end{lemma}

We denote by $\mathbf{SEM}$ the variety of all semigroups. The following claim gives a strong necessary condition for a semigroup variety to be a modular element in the lattice $\mathbb{SEM}$.

\begin{proposition}
\label{nec}
If $\mathbf V$ is a modular  element of the lattice $\mathbb{SEM}$ then either $\mathbf V=\mathbf{SEM}$ or $\mathbf V=\mathbf{M\vee N}$ where $\mathbf M$ is one of the varieties $\mathbf T$ or $\mathbf{SL}$, while  $\mathbf N$ is a nil-variety.\qed
\end{proposition}

This proposition was proved (in slightly weaker form and in some other terminology) in~\cite[Proposition~1.6]{Jezek-McKenzie-93}. A deduction of Proposition~\ref{nec} from~\cite[Proposition~1.6]{Jezek-McKenzie-93} was given explicitly in~\cite[Proposition~2.1]{Vernikov-07}. A direct and transparent proof of Proposition~\ref{nec} not depending on a technique from~\cite{Jezek-McKenzie-93} is provided in~\cite{Shaprynskii-12}. Note that an essetially stronger necessary condition for a semigroup variety to be a modular element in the lattice $\mathbb{SEM}$ than Proposition~\ref{nec} is given by the author in~\cite[Theorem~2.5]{Vernikov-07} but this stronger results will not be used below.

The following claim can be easily deduced from~\cite[Lemma~1]{Sapir-Sukhanov-81}.

\begin{lemma}
\label{split}
If a nil-variety of semigroups $\mathbf N$ satisfies an identity of the form
\begin{equation}
\label{x_1x_2...x_n=v}
x_1x_2\cdots x_n\approx\mathbf v
\end{equation}
then either $\mathbf N$ satisfies the identity
\begin{equation}
\label{nilpotence}
x_1x_2\cdots x_n\approx 0
\end{equation}
or the identity~\eqref{x_1x_2...x_n=v} is permutational.\qed
\end{lemma}

Let $n$ be a natural number. We denote by $S_n$ the full permutation group on the set $\{1,2,\dots,n\}$. If $1\le i\le n$ then we denote by $\Stab_n(i)$ the set of all permutations $\pi\in S_n$ with $i\pi=i$. Obviously, $\Stab_n(i)$ is a subgroup in $S_n$. Moreover, it is well known that $\Stab_n(i)$ is a maximal proper subgroup in $S_n$. Let $T$ be the trivial group, $T_{ij}$ be the group generated by the transposition $(ij)$, $C_{ijk}$ and $C_{ijk\ell}$ be the groups generated by the cycles $(ijk)$ and $(ijk\ell)$ respectively, $P_{ij,k\ell}$ be the group generated by the disjoint transpositions $(ij)$ and $(k\ell)$, $A_n$ be the alternative subgroup of $S_n$ and $V_4$ be the Klein four-group. The subgroup lattice of the group $G$ is denoted by $\Sub(G)$. We need to know the structure of the lattices $\Sub(S_3)$ and $\Sub(S_4)$.  It is generally known and easy to check that the first of these two lattices has the form shown in Fig.~\ref{Sub(S_3)}. Direct routine calculations allow to verify that the lattice $\Sub(S_4)$ is as shown in Fig.~\ref{Sub(S_4)}. 

\begin{figure}[tbh]
\unitlength=1mm
\special{em:linewidth 0.4pt}
\linethickness{0.4pt}
\begin{picture}(73,29)
\put(6,14){\circle*{1.33}}
\put(21,14){\circle*{1.33}}
\put(36,4){\circle*{1.33}}
\put(36,24){\circle*{1.33}}
\put(51,14){\circle*{1.33}}
\put(66,14){\circle*{1.33}}
\gasset{AHnb=0}
\drawline(36,4)(6,14)(36,24)(21,14)(36,4)(51,14)(36,24)(66,14)(36,4)
\put(5,14){\makebox(0,0)[rc]{$T_{12}$}}
\put(20,14){\makebox(0,0)[rc]{$T_{13}$}}
\put(52,14){\makebox(0,0)[lc]{$T_{23}$}}
\put(67,14){\makebox(0,0)[lc]{$C_{123}$}}
\put(36,27){\makebox(0,0)[cc]{$S_3$}}
\put(36,1){\makebox(0,0)[cc]{$T$}}
\end{picture}
\caption{The lattice $\Sub(S_3)$}
\label{Sub(S_3)}
\end{figure}
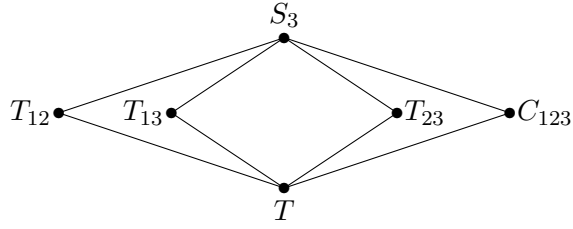

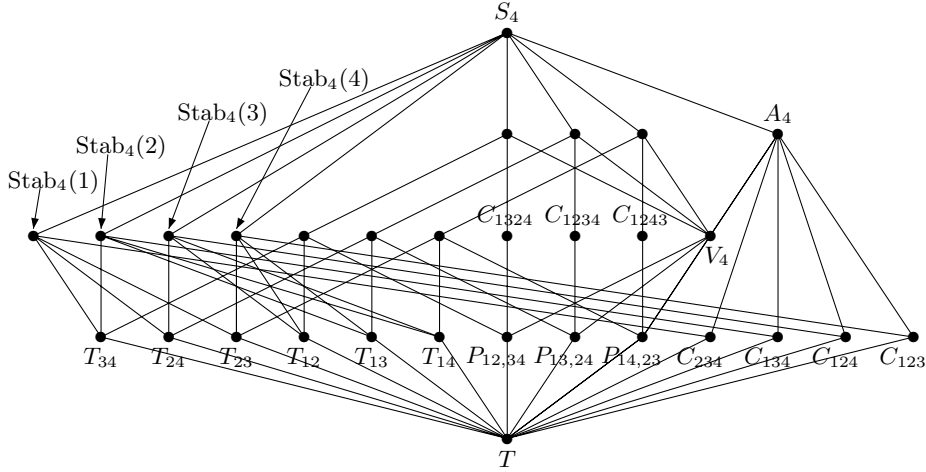
\begin{figure}[tbh]
\unitlength=0.9mm
\special{em:linewidth 0.4pt}
\linethickness{0.4pt}
\begin{picture}(135,70)
\put(73,5){\circle*{1.4}}
\put(13,20){\circle*{1.4}}
\put(23,20){\circle*{1.4}}
\put(33,20){\circle*{1.4}}
\put(43,20){\circle*{1.4}}
\put(53,20){\circle*{1.4}}
\put(63,20){\circle*{1.4}}
\put(73,20){\circle*{1.4}}
\put(83,20){\circle*{1.4}}
\put(93,20){\circle*{1.4}}
\put(103,20){\circle*{1.4}}
\put(113,20){\circle*{1.4}}
\put(123,20){\circle*{1.4}}
\put(133,20){\circle*{1.4}}
\put(3,35){\circle*{1.4}}
\put(13,35){\circle*{1.4}}
\put(23,35){\circle*{1.4}}
\put(33,35){\circle*{1.4}}
\put(43,35){\circle*{1.4}}
\put(53,35){\circle*{1.4}}
\put(63,35){\circle*{1.4}}
\put(73,35){\circle*{1.4}}
\put(83,35){\circle*{1.4}}
\put(93,35){\circle*{1.4}}
\put(103,35){\circle*{1.4}}
\put(73,50){\circle*{1.4}}
\put(83,50){\circle*{1.4}}
\put(93,50){\circle*{1.4}}
\put(113,50){\circle*{1.4}}
\put(73,65){\circle*{1.4}}
\gasset{AHnb=0}
\drawline(73,5)(13,20)(3,35)(23,20)(73,5)(33,20)(3,35)(73,65)(73,5)(83,20)(103,35)(93,20)(73,5)(103,20)(113,50)(93,20)(73,5)(123,20)(113,50)(133,20)(73,5)(43,20)(23,35)(73,65)(13,35)(13,20)(73,50)(103,35)(83,50)(73,65)(93,50)(103,35)(113,50)(73,65)(33,35)(33,20)(93,50)(93,20)(63,35)(63,20)(23,35)(123,20)
\drawline(23,35)(23,20)(83,50)(83,20)(53,35)(53,20)(73,5)(63,20)(13,35)(53,20)(33,35)(43,20)(43,35)(73,20)(103,35)
\drawline(33,20)(33,35)(133,20)
\drawline(13,35)(113,20)(73,5)
\drawline(113,20)(113,50)
\drawline(3,35)(103,20)
\footnotesize
\put(113,53){\makebox(0,0)[cc]{$A_4$}}
\put(128,17){\makebox(0,0)[lc]{$C_{123}$}}
\put(87,38){\makebox(0,0)[rc]{$C_{1234}$}}
\put(118,17){\makebox(0,0)[lc]{$C_{124}$}}
\put(97,38){\makebox(0,0)[rc]{$C_{1243}$}}
\put(77,38){\makebox(0,0)[rc]{$C_{1324}$}}
\put(108,17){\makebox(0,0)[lc]{$C_{134}$}}
\put(98,17){\makebox(0,0)[lc]{$C_{234}$}}
\put(76,17){\makebox(0,0)[rc]{$P_{12,34}$}}
\put(86,17){\makebox(0,0)[rc]{$P_{13,24}$}}
\put(96,17){\makebox(0,0)[rc]{$P_{14,23}$}}
\put(73,68){\makebox(0,0)[cc]{$S_4$}}
\put(6,43){\makebox(0,0)[cc]{$\Stab_4(1)$}}
\put(16,48){\makebox(0,0)[cc]{$\Stab_4(2)$}}
\put(31,53){\makebox(0,0)[cc]{$\Stab_4(3)$}}
\put(46,58){\makebox(0,0)[cc]{$\Stab_4(4)$}}
\put(73,2){\makebox(0,0)[cc]{$T$}}
\put(43,17){\makebox(0,0)[cc]{$T_{12}$}}
\put(53,17){\makebox(0,0)[cc]{$T_{13}$}}
\put(63,17){\makebox(0,0)[cc]{$T_{14}$}}
\put(33,17){\makebox(0,0)[cc]{$T_{23}$}}
\put(23,17){\makebox(0,0)[cc]{$T_{24}$}}
\put(13,17){\makebox(0,0)[cc]{$T_{34}$}}
\put(104,32){\makebox(0,0)[cc]{$V_4$}}
\gasset{AHnb=1}
\drawline(4,42)(3,36)
\drawline(14,47)(13,36)
\drawline(29,52)(23,36)
\drawline(44,57)(33,36)
\end{picture}
\caption{The lattice  $\Sub(S_4)$}
\label{Sub(S_4)}
\end{figure}

We need the following

\begin{lemma}
\label{subgroups cancellable}
A subgroup $G$ of the group $S_n$ is a cancellable element of the lattice $\Sub(S_n)$ if and only if either $G=T$ or $G=S_n$.
\end{lemma}

\begin{proof}
If $n\le 2$ then $S_n$ does not contain subgroups differ from $T$ and $S_n$. If $n=3$ then the desirable conclusion immediately follows from Fig.~\ref{Sub(S_3)}. Let now $n\ge 4$ and $G$ be a non-singleton proper subgroup of $S_n$. Suppose that $G$ is a cancellable and therefore, modular element of $\Sub(S_n)$. If $n=4$ then $G\supseteq V_4$ by~\cite[Proposition~3.8]{Jezek-81}, while if $n\ge 5$ then $G\supseteq A_n$ by~\cite[Proposition~3.1]{Jezek-81}. Clearly, it suffices to verify that there are at least two complements to $G$ in $\Sub(S_n)$. Suppose that $G\supseteq A_n$. Then $G=A_n$ because $A_n$ is a maximal proper subgroup in $S_n$. Then all subgroups of the form $T_{ij}$ are complements to $G$ in $\Sub(S_n)$. It remains to consider the case when $n=4$, $V_4\subseteq G\subset S_4$ and $G\ne A_4$. Fig.~\ref{Sub(S_4)} implies that either $G=V_4$ or $G=V_4\vee P_{ij,k\ell}$ for some disjoint transpositions $(ij)$ and $(k\ell)$. If $G=V_4$ then all subgroups of the form $\Stab_4(i)$ are complements to $G$ in $\Sub(S_4)$. Finally, if $G=V_4\vee P_{ij,k\ell}$ then subgroups $T_{ik}$ and $T_{j\ell}$ are complements to $G$ in $\Sub(S_4)$. 
\end{proof}

\section{The proof of Theorem~\ref{canc-permut-3}}
\label{proof}

For a semigroup variety \textbf V, we denote by $\Perm_n(\mathbf V)$ the set of all permutations $\pi\in S_n$ such that \textbf V satisfies the identity~\eqref{permut id}. Clearly, $\Perm_n(\mathbf V)$ is a subgroup in $S_n$. We need the following

\begin{lemma}
\label{modular from V to Perm_n(V)}
\!\!\textup{(\!\!\cite[Corollary~4.3]{Vernikov-07})} If a semigroup variety $\mathbf V$ is a modular element of the lattice $\mathbb{SEM}$ and $n$ is a positive integer then the group $\Perm_n(\mathbf V)$ is a modular element of the lattice $\Sub(S_n)$.\qed
\end{lemma}

A word $\mathbf u$ is called \emph{linear} if every letter occurs in $\mathbf u$ at most once. If $\mathbf u$ is a word then we denote by $\ell(\mathbf u)$ the length of $\mathbf u$ and by $\con(\mathbf u)$ the set of letters occurring in $\mathbf u$.  For brevity, we will denote the identity~\eqref{permut id} by $p_n[\pi]$. If $\mathbf w$ is a word, $\con(\mathbf w)=\{x_1,x_2,\dots,x_n\}$ and $\xi\in S_n$ then we denote by $\xi[\mathbf w]$ the word obtained from $\mathbf w$ by the substitution $x_i\mapsto x_{i\xi}$ for all $i=1,2,\dots,n$.

\begin{proposition}
\label{cancellable from V to Perm_n(V)}
If a semigroup variety $\mathbf V$ is a cancellable element of the lattice $\mathbb{SEM}$ and $n$ is a positive integer then the group $\Perm_n(\mathbf V)$ is a cancellable element of the lattice $\Sub(S_n)$.
\end{proposition}

\begin{proof}
Clearly, $\mathbf V$ is a modular  element of the lattice $\mathbb{SEM}$. By Proposition~\ref{nec}, either $\mathbf V=\mathbf{SEM}$ or $\mathbf V=\mathbf{M\vee N}$ where $\mathbf M$ is one of the varieties $\mathbf T$ or $\mathbf{SL}$, while  $\mathbf N$ is a nil-variety. It is evident that $\Perm_n(\mathbf{SEM})=T$ and $\Perm_n(\mathbf{SL\vee N})=\Perm_n(\mathbf N)$ for any $n$. Since $T$ is a cancellable element of $\Sub(S_n)$, we can assume that $\mathbf V=\mathbf N$. In particular, $\mathbf V$ is a nil-variety.

Put $V=\Perm_n(\mathbf V)$ for brevity. Suppose that $V$ is not a cancellable element of the lattice $\Sub(S_n)$. Then there are subgroups $X_1$ and $X_2$ of the group $S_n$ such that $V\vee X_1=V\vee X_2$ and $V\wedge X_1=V\wedge X_2$ but $X_1\ne X_2$. For $i=1,2$, we denote by $\mathbf X_i$ the variety given by the identity $x_1x_2\cdots x_{n+1}\approx 0$, all identities of the form $\mathbf w\approx 0$ where $\mathbf w$ is a word of length $n$ depending on $<n$ letters and all identities of the form $p_n[\pi]$ where $\pi\in X_i$. It is clear that $\mathbf X_1\ne\mathbf X_2$. 

One can verify that none of the varieties $\mathbf V$, $\mathbf X_1$ and $\mathbf X_2$ satisfies the identity~\eqref{nilpotence}. Indeed, suppose that this identity holds in $\mathbf V$. Then $\mathbf V$ satisfies all permutational identities of length $n$, whence $V=S_n$. But then
$$
X_1=V\wedge X_1=V\wedge X_2=X_2,
$$
contradicting the choice of the groups $X_1$ and $X_2$. Suppose now that the identity~\eqref{nilpotence} holds in $\mathbf X_1$. Then $X_1=S_n\supseteq X_2$. Lemma~\ref{modular from V to Perm_n(V)} implies that $V$ is a modular element of the lattice $\Perm_n(\mathbf V)$. Therefore,
\begin{align*}
X_1={}&(V\vee X_1)\wedge X_1&&\text{by the absorbtion law}\\
={}&(V\vee X_2)\wedge X_1&&\text{because }V\vee X_1=V\vee X_2\\
={}&(V\wedge X_1)\vee X_2&&\text{because }X_1\supseteq X_2\text{ and }V\text{ is modular in }\Perm_n(\mathbf V)\\
={}&(V\wedge X_2)\vee X_2&&\text{because }V\wedge X_1=V\wedge X_2\\
={}&X_2&&\text{by the absorbtion law},
\end{align*}
contradicting the choice of the groups $X_1$ and $X_2$ again. Analogous arguments show that the identity~\eqref{nilpotence} is not satisfied by $\mathbf X_2$.

Let now $\mathbf{u\approx v}$ be an arbitrary identity that holds in $\mathbf{V\vee X}_1$. We are going to verify that this identity holds in $\mathbf{V\vee X}_2$. Since $\mathbf{u\approx v}$ holds in $\mathbf V$, it suffices to verify that it holds in $\mathbf X_2$.  The identity  $\mathbf{u\approx v}$ holds in $\mathbf V$ and $\mathbf X_1$. If $\ell(\mathbf u),\ell(\mathbf v)\ge n+1$ then $\mathbf u\approx 0\approx \mathbf v$ in $\mathbf X_2$. Thus, we can assume without loss of generality that $\ell(\mathbf u)\le n$. On the other hand, the definition of the variety $\mathbf X_1$ and the fact that $\mathbf{u\approx v}$ holds in $\mathbf X_1$ imply that $\ell(\mathbf u),\ell(\mathbf v)\ge n$. In particular, $\ell(\mathbf u)=n$.

Suppose that the word $\mathbf u$ is linear. By Lemma~\ref{split}, either the identity $\mathbf{u\approx v}$ is permutational or the variety $\mathbf{V\vee X_1}$ satisfies the identity $\mathbf u\approx 0$. But we have proved above that the second case is impossible. Therefore, $\mathbf{u\approx v}$ is an identity of the form $p_n[\pi]$. Since it holds both in $\mathbf V$ and $\mathbf X_1$, we have $\pi\in V\wedge X_1=V\wedge X_2$. Hence $\pi\in X_2$, and therefore the identity $\mathbf{u\approx v}$ holds in $\mathbf X_2$. 

It remains to consider the case when the word $\mathbf u$ is non-linear. Then it depends on $<n$ letters. Therefore, $\mathbf u\approx 0$ in the varieties $\mathbf X_1$ and $\mathbf X_2$. If $\ell(\mathbf v)>n$ or $\mathbf v$ is a word of length $n$ depending on $<n$ letters then $\mathbf v\approx 0$ in $\mathbf X_2$. Then $\mathbf{u\approx v}$ holds in $\mathbf X_2$. Finally, if $\ell(\mathbf v)=n$ and $\mathbf v$ depends on $n$ letters then the word $\mathbf v$ is linear and we can complete our considerations by the same arguments as in the previous paragraph.

Thus, if the identity $\mathbf{u\approx v}$ holds in $\mathbf{V\vee X}_1$ then it holds in $\mathbf{V\vee X}_2$ too. This means that $\mathbf{V\vee X}_2\subseteq\mathbf{V\vee X}_1$. The inverse inclusion can be verified analogously, whence $\mathbf{V\vee X}_1=\mathbf{V\vee X}_2$.

Let now $\mathbf{u\approx v}$ be an arbitrary identity that holds in the variety $\mathbf{V\wedge X}_1$. We aim to verify that it holds in $\mathbf{V\wedge X}_2$. Let the sequence of words
\begin{equation}
\label{deduction}
\mathbf u=\mathbf u_0,\mathbf u_1,\dots,\mathbf u_k=\mathbf v
\end{equation}
be the shortest deduction of the identity $\mathbf{u\approx v}$ from the identities of the varieties $\mathbf V$ and $\mathbf X_1$. This means that, for any $i=0,1,\dots,k-1$, the identity $\mathbf u_i\approx\mathbf u_{i+1}$ holds in one of the varieties $\mathbf V$ and $\mathbf X_1$ and the situation when one of the varieties $\mathbf V$ and $\mathbf X_1$ satisfies the identities $\mathbf u_i\approx\mathbf u_{i+1}\approx\mathbf u_{i+2}$ for some $0\le i\le k-2$  is impossible. 

Suppose that there is an index $i$ such that $\mathbf u_i$ a linear word of length $n$. If $i>0$ then Lemma~\ref{split} implies that either $\mathbf u_{i-1}$ is a linear word and $\con(\mathbf u_{i-1})=\con(\mathbf u_i)$ or one of the varieties $\mathbf V$ or $\mathbf X_1$ satisfies the identity~\eqref{nilpotence}. Analogously, if $i<k$ then either $\mathbf u_{i+1}$ is a linear word and $\con(\mathbf u_i)=\con(\mathbf u_{i+1})$ or one of the varieties $\mathbf V$ and $\mathbf X_1$ satisfies the identity~\eqref{nilpotence}. But we have seen above that the last identity fails in both  the varieties $\mathbf V$ and $\mathbf X_1$. Therefore, the words adjacent to $\mathbf u_i$ in the sequence~\eqref{deduction} are linear words of length $n$ depending on the same letters as $\mathbf u_i$. By the trivial induction, this means that all the words $\mathbf u_0$, $\mathbf u_1$, \dots, $\mathbf u_k$ are linear words of length $n$ depending on the same letters. We may assume without loss of generality that $\con(\mathbf u)=\{x_1,x_2,\dots,x_n\}$. There are permutations $\pi_0,\pi_1,\dots,\pi_{k-1}\in S_n$ such that $\mathbf u_i=\pi_i[\mathbf u_{i-1}]$ for each $i=1,2,\dots,k$. Clearly, $\pi_i\in V$ [respectively $\pi_i\in X_1$] whenever the identity $\mathbf u_{i-1}\approx\mathbf u_i$ holds in the variety $\mathbf V$ [respectively $\mathbf X_1$]. Put $\pi=\pi_0\pi_1\cdots\pi_{k-1}$. Then $\pi\in V\vee X_1=V\vee X_2$. Therefore, there are permutations $\sigma_0,\sigma_1,\dots,\sigma_{m-1}\in S_n$ such that $\pi=\sigma_0\sigma_1\cdots\sigma_{m-1}$ and, for all $i=0,1,\dots,m-1$, the permutation $\sigma_i$ lies in either $V$ or $X_2$. Put $\mathbf v_0=\mathbf u$ and $\mathbf v_i=\sigma_i[\mathbf u_{i-1}]$ for each $i=1,2,\dots,m$. Obviously, $\mathbf v_m=\mathbf v$ and, for any $i=0,1,\dots,m-1$, the identity $\mathbf v_i\approx\mathbf v_{i+1}$ holds in one of the varieties $\mathbf V$ or $\mathbf X_2$. Therefore, the identity $\mathbf{u\approx v}$ holds in the variety $\mathbf{V\wedge X}_2$.

Suppose now that there are no linear words of length $n$ among the words $\mathbf u_0$, $\mathbf u_1$, \dots, $\mathbf u_k$. The definition of the variety $\mathbf X_1$ shows that if this variey satisfies an identity $\mathbf p=\mathbf q$ then $\ell(\mathbf p),\ell(\mathbf q)\ge n$. This means that:
\begin{itemize}
\item[$\bullet$] $\ell(\mathbf u_i)\ge n$ for all $i=1,2,\dots,k-1$;
\item[$\bullet$] either $\ell(\mathbf u_0)\ge n$ or the identity $\mathbf u_0\approx\mathbf u_1$ holds in $\mathbf V$;
\item[$\bullet$] either $\ell(\mathbf u_k)\ge n$ or the identity $\mathbf u_{k-1}\approx\mathbf u_k$ holds in $\mathbf V$.
\end{itemize}
In view of we say above, if $\ell(\mathbf u_i)=n$ for some $0\le i\le k$ then the word $\mathbf u_i$ is non-linear. Then the sequence~\eqref{deduction} is a deduction of the identity $\mathbf{u\approx v}$ from the identities of the varieties $\mathbf V$ and $\mathbf X_2$. Thus, we have again that  the identity $\mathbf{u\approx v}$ holds in the variety $\mathbf{V\wedge X}_2$.

We prove that if the identity $\mathbf{u\approx v}$ holds in $\mathbf{V\wedge X}_1$ then it holds in $\mathbf{V\wedge X}_2$. Therefore, $\mathbf{V\wedge X}_2\subseteq\mathbf{V\wedge X}_1$. The inverse inclusion can be verified analogously, whence $\mathbf{V\wedge X}_1=\mathbf{V\wedge X}_2$.

Thus, $\mathbf{V\vee X}_1=\mathbf{V\vee X}_2$ and $\mathbf{V\wedge X}_1=\mathbf{V\wedge X}_2$. Since $\mathbf V$ is cancellable in $\mathbb{SEM}$, this implies that $\mathbf X_1=\mathbf X_2$. But then $X_1=X_2$, contradicting the choice of the groups $X_1$ and $X_2$.
\end{proof}

Proposition~\ref{cancellable from V to Perm_n(V)} and Lemma~\ref{subgroups cancellable} immediately imply the following 

\begin{proposition}
\label{alternative for cancellable}
Let $\mathbf V$ be a cancellable element of the lattice $\mathbb{SEM}$ and $n$ be a natural number. If $\mathbf V$ satisfies some permutational identity of length $n$ then it satisfies all such identities.\qed
\end{proposition}

Proposition~\ref{alternative for cancellable} and~\cite[Theorem~4.5]{Vernikov-07} imply the following

\begin{corollary}
\label{non-linear = 0}
Let $\mathbf V$ be a cancellable element of the lattice $\mathbb{SEM}$. If $\mathbf V$ satisfies some permutational identity of length $n\ge 3$ then it satisfies all identities of the form $\mathbf u\approx 0$ where $\mathbf u$ is a word of length $n$ depending on $<n$ letters.\qed
\end{corollary}

We say that a semigroup variety has a \emph{degree} $n$ if all nilsemigroups in this variety nilpotent of degree $\le n$ and $n$ is the least number with such a property. 

\begin{proof}[Proof of Theorem~\textup{\ref{canc-permut-3}}. Necessity.]
 Let $\mathbf V$ be a semigroup variety that satisfies a permutational identity of length~3 and is a cancellable element of the lattice $\mathbb{SEM}$. Clearly, $\mathbf V$ is a modular element of $\mathbb{SEM}$. In view of Proposition~\ref{nec}, $\mathbf V=\mathbf{M\vee N}$ where $\mathbf M$ is one of the varieties $\mathbf T$ or $\mathbf{SL}$, while  $\mathbf N$ is a nil-variety. Proposition~\ref{alternative for cancellable} implies that $\mathbf V$ and therefore, $\mathbf N$ satisfies the identities~\eqref{xyz=yxz=xzy}. Now Corollary~\ref{non-linear = 0} applies with the conclusion that the identity~\eqref{xxy=0} holds in $\mathbf N$.

\medskip

\emph{Sufficiency.} Let $\mathbf V=\mathbf{M\vee N}$ where $\mathbf M$ is one of the  varieties $\mathbf T$ or $\mathbf{SL}$, while the variety $\mathbf N$ satisfies the identities~\eqref{xyz=yxz=xzy} and~\eqref{xxy=0}. We need to verify that $\mathbf V$ is cancellable. In view of Lemma~\ref{join with SL} we can assume that $\mathbf V=\mathbf N$. The identities~\eqref{xyz=yxz=xzy} imply $xyz\approx zyx$. Proposition~\ref{mod-permut-3} implies then that the variety $\mathbf N$ is a modular element of $\mathbb{SEM}$. Suppose that $\mathbf N$ is not a cancellable element of $\mathbb{SEM}$.  Let $\mathbf Y$ and $\mathbf Z$ be semigroup varieties such that $\mathbf N\vee\mathbf Y=\mathbf N\vee\mathbf Z$ and $\mathbf N\wedge\mathbf Y=\mathbf N\wedge\mathbf Z$. In view of Lemma~\ref{modular non-cancellable}, there is a variety $\mathbf N'$ such that
$$
\mathbf{N'\subseteq N},\,\mathbf{N'\vee Y}=\mathbf{N'\vee Z}=\mathbf{Y\vee Z}\text{ and }\mathbf{N'\wedge Y}=\mathbf{N'\wedge Z}.
$$
It is evident that the identities~\eqref{xyz=yxz=xzy} and~\eqref{xxy=0} implies
\begin{equation}
\label{xyx=yxx=0}
xyx\approx yx^2\approx 0.
\end{equation}
Being a subvariety of \textbf N, the variety $\mathbf N'$ satisfies the identities~\eqref{xyz=yxz=xzy},~\eqref{xxy=0} and~\eqref{xyx=yxx=0}. We aim to verify that $\mathbf Y=\mathbf Z$. This completes the proof because contradicts the claim that $\mathbf N$ is not a cancellable element of $\mathbb{SEM}$. By symmetry, it suffices to check that $\mathbf{Z\subseteq Y}$. Let $\mathbf{u\approx v}$ be an arbitrary identity that holds in $\mathbf Y$. It suffices to prove that this identity holds in $\mathbf Z$.

First of all we note that if the identity $\mathbf{u\approx v}$ holds in $\mathbf N'$ then it holds also in $\mathbf N'\vee\mathbf Y=\mathbf N'\vee\mathbf Z$  and therefore, in $\mathbf Z$. In particular, this allows us to assume that one of the words $\mathbf u$ and $\mathbf v$, say $\mathbf u$, does not equal to~0 in $\mathbf N'$. On the other hand, the fact that $\mathbf N'$ satisfies the identities~\eqref{xyz=yxz=xzy},~\eqref{xxy=0} and~\eqref{xyx=yxx=0} implies that every non-linear word except $x^2$ equals to~0 in $\mathbf N$ and therefore, in $\mathbf N'$. Thus, we can assume that either the word $\mathbf u$ is linear or $\mathbf u=x^2$. The latter case can be considered by literally repeating arguments from the consideration of the same case in the proof of sufficiency of Theorem~1.1 in~\cite{Gusev-Skokov-Vernikov-18+}. Suppose now that the word $\mathbf u$ is linear. The case when the identity $\mathbf{u\approx v}$ is not permutational also can be considered by literally repeating arguments from the consideration of the same case in the proof of sufficiency of Theorem~1.1 in~\cite{Gusev-Skokov-Vernikov-18+}. Finally, suppose that the identity $\mathbf{u\approx v}$ is permutational. If its length is greater than 2 then the identity $\mathbf{u\approx v}$ holds in $\mathbf N$ and therefore, in $\mathbf N'$. We reduce the proof to the case then the identity $\mathbf{u\approx v}$ is the commutative law. 

Thus, the variety $\mathbf Y$ is commutative. The identity $xy\approx yx$ holds in the variety $\mathbf N'\wedge\mathbf Y=\mathbf N'\wedge\mathbf Z$. Therefore, there is a deduction of this identity from identities of the varieties $\mathbf N'$ and $\mathbf Z$. In particular, one of these varieties satisfies an identity of the form $xy\approx\mathbf w$. Suppose that this identity holds in $\mathbf N'$. Then Lemma~\ref{split} implies that $\mathbf N'$ either is commutative or satisfies the identity $xy\approx 0$. Clearly, $\mathbf N'$ is commutative always. Therefore, the commutative law holds in the variety $\mathbf N'\vee\mathbf Y=\mathbf N'\vee\mathbf Z$. Then the variety $\mathbf Z$ is commutative, and we are done.

It remains to consider the case when the identity $xy\approx\mathbf w$ holds in $\mathbf Z$. Lemma~\ref{split} implies that in this case either the variety $\mathbf Z$ is commutative or all nilsemigroups in $\mathbf Z$ are semigroups with zero multiplication. In the former case we are done. It remains to consider the latter case. In other words, $\mathbf Z$ is a variety of degree $\le 2$. Repeating literally arguments of the proof of~\cite[Lemma~3.1]{Gusev-Skokov-Vernikov-18+}, we can verify that the varieties $\mathbf Y$ and $\mathbf Z$ has the same degree. Therefore, $\mathbf Y$ is a variety of degree $\le 2$ too. By~\cite[Lemma~3]{Golubov-Sapir-82} or~\cite[Proposition~2.11]{Vernikov-08} this means that each of the varieties $\mathbf Y$ and $\mathbf Z$ saisfies one of the identities
\begin{align}
\label{left} xy&\approx x^{m+1}y,\\
\label{right} xy&\approx xy^{m+1},\\
\label{center} xy&\approx (xy)^{m+1}
\end{align}
for some natural $m$. Clearly, if $\mathbf Y$ [respectively $\mathbf Z$] satisfies one of these identities with $m=r$ [respectively $m=s$] then both the varieties satisfy corresponding identities with $m=rs$. Thus, we can assume that values of the parameter $m$ for $\mathbf Y$ and $\mathbf Z$ coincide. Further considerations are divided into seven cases.

\medskip

\emph{Case} 1: the variety $\mathbf Z$ satisfies the identity~\eqref{center}. Being commutative, the variety $\mathbf Y$ satisfies the identity $(xy)^{m+1}\approx(yx)^{m+1}$. Therefore, this identity holds in $\mathbf N'\vee\mathbf Y=\mathbf N'\vee\mathbf Z$. In particular, it holds in $\mathbf Z$. Hence $\mathbf Z$ satisfies the identities $xy\approx(xy)^{m+1}\approx(yx)^{m+1}\approx yx$.

\medskip

\emph{Case} 2: the varieties $\mathbf Y$ and $\mathbf Z$ satisfy the identity~\eqref{left}. Here $\mathbf Y$ satisfies the identities $x^{m+1}y\approx xy\approx yx\approx y^{m+1}x$. Therefore, the identity $x^{m+1}y\approx y^{m+1}x$ holds in $\mathbf N'\vee\mathbf Y=\mathbf N'\vee\mathbf Z$. In particular, it holds in $\mathbf Z$. Hence $\mathbf Z$ satisfies the identities $xy\approx x^{m+1}y\approx y^{m+1}x\approx yx$. 

\medskip

\emph{Case} 3: the varieties $\mathbf Y$ and $\mathbf Z$ satisfy the identity~\eqref{right}. This case is dual to the previous one.

\medskip

\emph{Case} 4: the varieties $\mathbf Y$ and $\mathbf Z$ satisfy the identities~\eqref{left} and~\eqref{right} respectively. Here $\mathbf Y$ satisfies the identities $xy^{m+1}\approx y^{m+1}x\approx yx\approx xy\approx x^{m+1}y\approx yx^{m+1}$. Therefore, the identity $xy^{m+1}\approx yx^{m+1}$ holds in $\mathbf N'\vee\mathbf Y=\mathbf N'\vee\mathbf Z$. In particular, it  holds in $\mathbf Z$. Hence $\mathbf Z$ satisfies the identities $xy\approx xy^{m+1}\approx yx^{m+1}\approx yx$.

\medskip

\emph{Case} 5: the varieties $\mathbf Y$ and $\mathbf Z$ satisfy the identities~\eqref{right} and~\eqref{left} respectively. This case is dual to the previous one.

\medskip

\emph{Case} 6: the varieties $\mathbf Y$ and $\mathbf Z$ satisfy the identities~\eqref{center} and~\eqref{left} respectively. Here $\mathbf Y$ satisfies the identity $x^{m+1}y\approx(x^{m+1}y)^{m+1}$. Therefore, this identity holds in $\mathbf N'\vee\mathbf Y=\mathbf N'\vee\mathbf Z$ and therefore, in $\mathbf Z$. Hence $\mathbf Z$ satisfies the identities $xy\approx x^{m+1}y\approx(x^{m+1}y)^{m+1}\approx(xy)^{m+1}$. We reduce the situation to the Case 1 considered above.

\medskip

\emph{Case} 7: the varieties $\mathbf Y$ and $\mathbf Z$ satisfy the identities~\eqref{center} and~\eqref{right} respectively. This case is dual to the previous one.
\end{proof}

\end{document}